\documentclass[11pt]{amsart}
\usepackage{amssymb}
\usepackage{amsmath}
\usepackage{amscd}
\usepackage{verbatim}
\usepackage{amssymb,amsfonts,pstricks,epsf}
\usepackage{graphicx}
\usepackage{epsfig}
\usepackage{color}
\usepackage{stmaryrd}
\usepackage{mathrsfs}
\usepackage{marginnote}
\usepackage{geometry}
\usepackage{mathtools}

\newtheorem{theorem}{Theorem}[section]
\newtheorem{lemma}[theorem]{Lemma}

\newcommand{\R}{\mathbb R}

\newcommand{\N}{\mathbb N}

\definecolor{plum}{rgb}{1.0, 0.0, 1.0}

\begin{document}

\title[]{A sharp estimate for Neumann eigenvalues  of the Laplace-Beltrami operator for domains in a hemisphere }

\author{Rafael D. Benguria, Barbara Brandolini, and Francesco Chiacchio}

\address{Rafael D. Benguria\\
Instituto de F\'\i sica,  Facultad de F\'\i sica\\
P. Universidad Cat\'olica de Chile\\
Casilla 306, Santiago 22, Chile
}
\email{rbenguri@fis.puc.cl}

\address{
Barbara Brandolini \\
Universit\`a degli Studi di Napoli ``Federico II''\\
Dipartimento di Ma\-te\-ma\-ti\-ca e Applicazioni ``R. Caccioppoli''\\
Complesso Monte S. Angelo - Via Cintia\\
80126 Napoli, Italia.
}
\email{brandolini@unina.it}

\address{
Francesco Chiacchio \\
Universit\`a degli Studi di Napoli ``Federico II''\\
Dipartimento di Ma\-te\-ma\-ti\-ca e Applicazioni ``R. Caccioppoli''\\
Complesso Monte S. Angelo - Via Cintia\\
80126 Napoli, Italia.
}
\email{francesco.chiacchio@unina.it}


\begin{abstract}
Here we prove an isoperimetric inequality for the harmonic  mean of the first $N-1$ non-trivial Neumann eigenvalues of the Laplace-Beltrami operator for domains contained in a hemisphere of $\mathbb{S}^N$.
\end{abstract}



\maketitle

\section{Introduction}
Let $\Omega$ be a bounded domain in $\R^N$ and let us consider the eigenvalues of the classical Neumann-Laplacian in $\Omega$,
$$
0=\mu_0(\Omega)<\mu_1(\Omega)\le\mu_2(\Omega)\le ...
$$
Isoperimetric inequalities for the $\mu_i$'s go back to the
classical theorem of Szeg\H{o} \cite{Sz} and Weinberger \cite{W}: 
\emph{the ball maximizes $\mu_1(\Omega)$ among all bounded
smooth domains $\Omega$ in  $\mathbb{R}^N$ having the same measure}.
Szeg\H{o}, using conformal maps, proved it for simply connected domains in
$\mathbb{R}^2$, while Weinberger introduced a method that allowed him to get
this result in full generality in $\mathbb{R}^N$. His technique has been adapted 
in different contexts  to establish isoperimetric results for combination of eigenvalues 
of the Laplacian with Dirichlet or Neumann boundary conditions (see e.g. \cite{AB92,AB00,Ban, BCdB,BH,C1,CdB,LS}). 
For further references see, e.g., the monographs \cite{C,H1,H2}  and the survey paper \cite{A}.
Actually, as well-known,  the conformal map technique used by Szeg\H{o} allows to prove the stronger 
inequality
\begin{equation}\label{s}
\frac{1}{\mu_1 (\Omega)} +\frac{1}{\mu_2 (\Omega)}
\geq 
\frac{2}{\mu_1 (\Omega^{\star})} ,
 \end{equation}
again for  simply connected domains in $\mathbb{R}^2$. Here and in the sequel, $\Omega^{\star}$
will denote the disk, or, more in general, the ball in $\mathbb{R}^N$ having the 
same measure  as $\Omega$. 
Inequality \eqref{s} is sharp since equality 
sign is achieved if and only if $\Omega$ is a disk. Later, in \cite{AB93}, the assumption of simply connectedness was removed.
In the same paper the authors conjectured that an inequality analogous to \eqref{s} holds true in $\mathbb{R}^N$, namely
\begin{equation*}
\frac{1}{\mu_1 (\Omega)} +....+\frac{1}{\mu_N (\Omega)}
\geq 
\frac{N}{\mu_1 (\Omega^{\star})}. 
 \end{equation*}
 Very recently, in \cite{WX} the authors made an important step toward the proof of this conjecture, by showing the following inequality
 \begin{equation*}
\frac{1}{\mu_1 (\Omega)} +....+\frac{1}{\mu_{N-1} (\Omega)}
\geq 
\frac{N-1}{\mu_1 (\Omega^{\star})}. 
 \end{equation*}
 
The aim of this manuscript is to prove an analogous result for the Laplace-Beltrami operator with Neumann boundary conditions.
 Precisely, we deal with non-trivial Neumann eigenvalues of an arbitrary domain $\Omega$ contained in a hemisphere of $\mathbb{S}^N$, defined by the following boundary value problem 
\begin{equation}\label{pb}
\left\{\begin{array}{ll}
-\Delta_{{\mathbb S}^N} u =\mu u &\mathrm{in } \>\Omega
\\ \\
\frac{\partial u}{\partial \nu}=0 & \mathrm{on }\> \partial \Omega,
\end{array}
\right.
\end{equation}
where $\nu$ is the unit normal to $\partial \Omega$. We still denote the eigenvalues of \eqref{pb} with $\mu_i(\Omega)$ and we intend them arranged in an increasing way, that is
$$
0=\mu_0(\Omega)<\mu_1(\Omega)\le\mu_2(\Omega)\le ...
$$
If we denote by $\{u_i\}_i$ a sequence of orthonormal set of eigenfunctions corresponding to $\mu_i(\Omega)$, then the following variational characterization holds true 
\begin{equation}\label{vc}
\mu_i(\Omega)=\min\left\{\dfrac{\displaystyle\int_\Omega |\nabla \phi|^2\, d\omega}{\displaystyle\int_\Omega \phi^2\, d\omega}:\> \phi\in H^1(\Omega)\setminus\{0\},\> \phi \in \mathrm{span}\left\{u_0,u_1,...,u_{i-1}\right\}^\perp\right\}.
\end{equation}
The analogous of the Szeg\H{o}-Weinberger result is already known and was proved in \cite{AB95}. Our main result is the following
\begin{theorem}\label{theo}
With the notation as above, 
\begin{equation}\label{ineq}
\sum_{i=1}^{N-1} \frac{1}{\mu_i(\Omega)} \ge \sum_{i=1}^{N-1} \frac{1}{\mu_i(D_\gamma)}
\end{equation}
where $D_\gamma$ is a geodesic ball contained in a hemisphere of $\mathbb{S}^N$ having the same $N$-volume as $\Omega$, and $\gamma$ is its radius. More precisely, $\gamma$ is determined by
$$
|\Omega| =N\omega_N \int_0^\gamma \sin^{N-1}t\, dt,
$$
where $\omega_N$ denotes the volume of the unit ball in $\R^N$. Equality sign holds in \eqref{ineq} if and only if $\Omega$ is a geodesic ball.
\end{theorem}
\section{Properties of the Neumann eigenvalues and eigenfunctions of a geodesic ball}
Let $D_\gamma$ be a geodesic ball on $\mathbb{S}^N$ having radius $\gamma$. We think to this geodesic ball as the set of points of $\mathbb{S}^N$ with angle from the positive $x_{N+1}$-axis less that $\gamma$, that is a polar cap. By standard separation of variables technique, we find that the eigenvalues of \eqref{pb}, with $\Omega=D_\gamma$, are the eigenvalues of the following one-dimensional problems
\begin{equation*}\label{eq}
\left\{
\begin{array}{ll}
-\dfrac{1}{\sin^{N-1}\theta} \dfrac{d}{d \theta}\left(\sin^{N-1}\theta \dfrac{dy}{d\theta}\right)+\dfrac{l(l+N-2)}{\sin^2\theta}y=\mu_{l,k}\,  y \quad \mathrm{in} \>(0,\gamma) 
\\ \\
y(0)\> \mathrm{finite}, y'(\gamma)=0
\end{array}
\right.
\end{equation*}
with $l\in \N_0, k\in \N$. Clearly, $\mu_1(D_\gamma)=\min\left\{\mu_{0,2},\mu_{1,1}\right\}$. In \cite{AB95} the authors show that $\mu_1(D_\gamma)=\mu_{1,1}$ at least if $\gamma \le \frac{\pi}{2}$. Hence, an eigenfunction $g$ (assumed positive) associated to $\mu_{1,1}=\mu_1(D_\gamma)$ satisfies
\begin{equation}\label{eqg}
\left\{\begin{array}{ll}
-g''-(N-1)\cot \theta\, g'+\dfrac{N-1}{\sin^2\theta}\,g=\mu\, g \quad \mathrm{in}\> (0,\gamma)
\\ \\
g(0)= g'(\gamma)=0.
\end{array}
\right.
\end{equation}
Multiplying the equation in \eqref{eqg} by $g$ and then integrating on $D_\gamma$ yields 
\begin{equation}\label{mu1}
\mu_1(D_\gamma)=\dfrac{\displaystyle\int_{D_\gamma} \left[g'(\theta)^2+(N-1)\, \frac{g(\theta)^2}{\sin^2\theta}\right]\,d\omega}{\displaystyle\int_{D_\gamma}g(\theta)^2\,d\omega}.
\end{equation}
The following properties are also proved in \cite{AB95}.
\begin{enumerate}
\item If $0<\gamma \le \frac \pi 2$, then $g'>0$ in $[0,\gamma)$, thus $g$ is strictly increasing in $[0,\gamma]$. 

\item $\mu_1(D_\gamma)$ is a strictly decreasing function of $\gamma$ for $0<\gamma\le \frac \pi 2$.

\item $\mu_1(D_\gamma)>N=\mu_1(D_{\pi/2})$ for $0<\gamma<\frac \pi 2$. 
\end{enumerate}
We also recall that $\mu_1(D_\gamma)$ is $N$-fold degenerate, that is
$$
\mu_1(D_\gamma)=\mu_2(D_\gamma)=...=\mu_N(D_\gamma).
$$
Now, define $G:\left[0,\frac \pi 2\right]\to [0,+\infty)$ by
\begin{equation}\label{G}
G(\theta)=\left\{
\begin{array}{ll}
g(\theta) & \theta \le \gamma
\\
g(\gamma) & \theta>\gamma.
\end{array}
\right.
\end{equation}
\begin{lemma}\label{lemma}
The function $\dfrac{G(\theta)}{\sin\theta}$ is strictly decreasing in $\left[0,\frac \pi 2\right]$.
\end{lemma}

\begin{proof}
By Taylor-Frobenius expansion we have
$G(\theta)=\theta-a\,\theta^3+o(\theta^3)$, where 
$$
a=\frac{\mu_1(D_\gamma)-\frac 2 3 (N-1)}{2N+4}>0.
$$
In order to get the claim it is enough to prove that
$$
W(\theta):=G'(\theta)-G(\theta)\,\cot \theta <0.
$$
Using the behavior of $G(\theta)$ near $\theta=0$ we have
$$
W(\theta) =\left(\frac 1 3 -2a\right)\theta^2+o(\theta^2)=\left(\frac{N-\mu_1(D_\gamma)}{N+2}\right)\theta^2+o(\theta^2).
$$
Property (3) implies that $W(\theta) <0$ close to 0. We also know that $W(\gamma)<0$. Assume by contradiction that $W(\theta)$ attained a positive maximum at a point $\tilde \theta \in (0,\gamma )$. Hence
$$
W(\tilde \theta)>0, \quad W'(\tilde \theta)=G''(\tilde \theta)- G'(\tilde \theta)\, \cot \tilde \theta+\frac{G(\tilde \theta)}{\sin^2 \tilde \theta}=0.
$$
Using the equation in \eqref{eqg} we gain
$$
N\left[G'(\tilde \theta)\, \cot \tilde \theta -\frac{G(\tilde\theta)}{\sin^2\tilde \theta}\right]=-\mu_1(D_\gamma)\, G(\tilde \theta),
$$
that is
$$
N\left[W(\tilde \theta)\, \cot \tilde \theta -G(\tilde \theta)\right]=-\mu_1(D_\gamma)\, G(\tilde \theta).
$$
Since we are assuming that $W(\tilde \theta)>0$, property (3) immediately gives a contradiction.

\end{proof}

%

\section{Some mathematical tools needed for the proof of Theorem \ref{theo}}
For the proof of our main result, Theorem \ref{theo}, it is convenient to parametrize the points of $\Omega$ in terms of the coordinates of their stereographic projection (see, for example, \cite{BC,G}). For a point $P\in \Omega$, we denote by $P'$ its stereographic projection from the South Pole $S$ onto the ``equator'' (as illustrated in Figure 1). 

\begin{figure}[h]
\centering 
\includegraphics[scale=0.35]{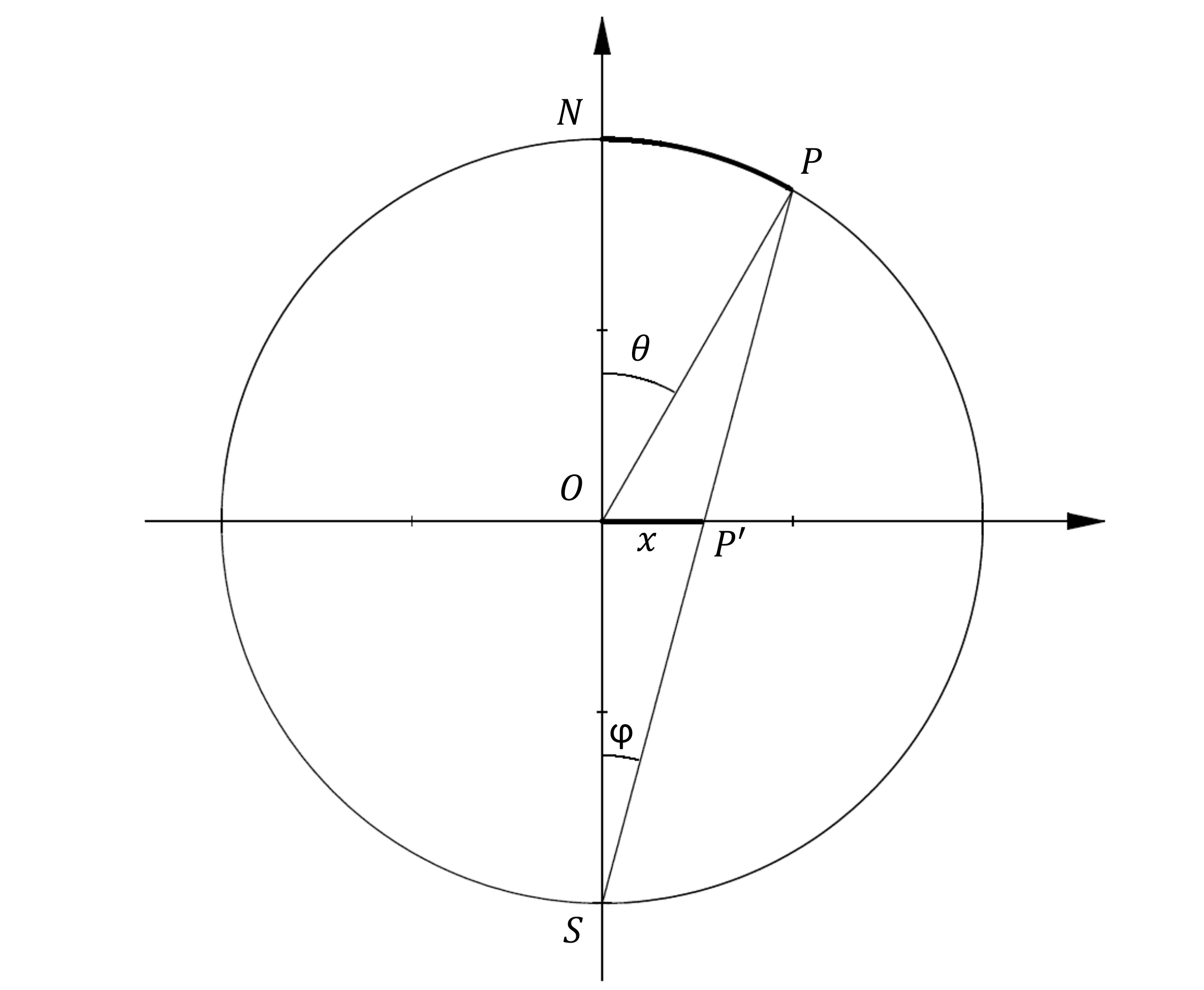}
\caption{Stereographic coordinates}
\end{figure}

For $P'$ we use cartesian coordinates $(x_1,x_2,...,x_N,0)$. We also use $s=\sqrt{\sum_{i=1}^N x_i^2}$, the euclidean distance from $P'$ to the origin $O$. As used we denote by $\theta$ the azimuthal angle, i. e. the angle between $ON$ and $OP$, where $N$ stands for the North Pole. Moreover we denote by $\varphi$ the angle between $SN$ and $SP$. It is clear that $\theta=2\varphi$ and $\tan \varphi=s$. Hence,
\begin{equation}\label{1}
\theta=2\arctan s
\end{equation}
from which we immediately get
\begin{equation}\label{2}
\frac{d\theta}{ds}=\frac{2}{1+s^2}=p(s),
\end{equation}
the conformal factor associated to the differential structure on $\mathbb{S}^N$. In terms of the conformal factor $p$ we can write
\begin{equation*}\label{3}
\nabla_{\mathbb{S}^N}=\frac{1}{p}\nabla_{\R^N},
\end{equation*}
where $\nabla_{\R^N}$ is the standard gradient on the equator. We also have
\begin{equation*}\label{4}
-\Delta_{\mathbb{S}^N}=-p^{-N}\mathrm{div}\left(p^{N-2}\nabla_{\R^N}u\right).
\end{equation*}
Finally, from the figure (or directly from \eqref{1} and \eqref{2}) we also have that
\begin{equation}\label{5}
\sin \theta=p\cdot s.
\end{equation}
In the sequel we also need to compute $\theta_{,i}:=\dfrac{\partial \theta}{\partial x_i}$. Using \eqref{2}, the definition of $s$ and the chain rule we have
\begin{equation*}
\theta_{,i}=\frac{\partial \theta}{\partial s}\cdot s_{,i}=p\, \frac{x_i}{s}, \quad i=1,...,N,
\end{equation*}
and
\begin{equation}\label{7}
\sum_{i=1}^N\theta_{,i}^2=p^2.
\end{equation}
With the notation introduced above, we define 
\begin{equation}\label{8}
\Phi_i(x)=G(\theta)\frac{x_i}{s},\quad i=1,...,N,
\end{equation}
where $G(\theta)$ is defined in \eqref{G}. In order to use $\Phi_i$ as test function in \eqref{vc}, we need the following orthogonality conditions
\begin{equation}\label{ort}
\int_\Omega \Phi_i\, u_j\, d\omega=0, \quad i=1,...,N, \> j=0,...,i-1,
\end{equation}
where, as we said,   $u_j$ is an eigenfunction corresponding to $\mu_j(\Omega)$. To fulfill these conditions we need a special ``orientation'' of the sphere $\mathbb{S}^N$. When $j=0$, conditions \eqref{ort} can be immediately deduced from Theorem 2.1 in \cite{AB95} via the following identity
$$
\int_\Omega \Phi_i\, d\omega=\int_\Omega G(\theta)\, \frac{x_i}{s}\, d\omega=\int_\Omega \frac{G(\theta)}{\sin \theta}\, y_i \, d\omega,
$$ 
choosing $\tilde G(\theta)=\dfrac{G(\theta)}{\sin \theta}$. When $j>0$, conditions \eqref{ort} can be proved arguing in an analogous way as in the proof of Theorem 2.1 in \cite{AB93}.

\section{Proof of Theorem \ref{theo}}

Recalling the definition of $\Phi_i$ given in  \eqref{8}, we get 
\begin{equation}\label{9}
\left(\nabla \Phi_i\right)_j \equiv \Phi_{i,j}=G'(\theta)\, p\, \frac{x_ix_j}{s^2}+G(\theta)\,\frac{\delta_{ij}}{s}-G(\theta)\frac{x_ix_j}{s^3},\quad j=1,...,N.
\end{equation}
Using \eqref{7}, the definition of $s$ and \eqref{9} we have
\begin{equation}\label{10}
\frac{1}{p^2}\left|\nabla \Phi_i\right|^2=G'(\theta)^2 \frac{x_i}{s^2}+G(\theta)^2 \frac{1}{s^2p^2}-G(\theta)^2 \frac{x_i^2}{p^2 s^4}.
\end{equation}
Hence, from \eqref{5} and \eqref{10},
\begin{equation*}
\sum_{i=1}^N\left|\left|\nabla_{\mathbb{S}^N}\Phi_i\right|\right|^2=\frac{1}{p^2}\sum_{i=1}^N\left|\nabla \Phi_i\right|^2=G'(\theta)^2 +G(\theta)^2 \frac{N-1}{s^2p^2}=G'(\theta)^2 +G(\theta)^2 \frac{N-1}{\sin^2\theta}.
\end{equation*}
Using $\Phi_i$ as test function in the variational characterization \eqref{vc} of $\mu_i(\Omega)$, and taking into account the orthogonality conditions \eqref{ort}, we get
\begin{eqnarray}
\int_\Omega \Phi_i^2\, d\omega &\le& \frac{1}{\mu_i(\Omega)} \int_\Omega G'(\theta)^2 \,\frac{x_i^2}{s^2}\, d\omega+\frac{1}{\mu_i(\Omega)}\int_\Omega  \frac{G(\theta)^2}{\sin^2\theta}\left(1-\frac{x_i^2}{s^2}\right)\,d\omega \notag
\\
&=& \frac{1}{\mu_i(\Omega)} \int_{\Omega \cap D_\gamma}G'(\theta)^2 \,\frac{x_i^2}{s^2}\, d\omega+\frac{1}{\mu_i(\Omega)}\int_\Omega  \frac{G(\theta)^2}{\sin^2\theta}\left(1-\frac{x_i^2}{s^2}\right)\,d\omega \notag
\\
&\le& \frac{1}{\mu_i(\Omega)} \int_{ D_\gamma}G'(\theta)^2 \,\frac{x_i^2}{s^2}\, d\omega+\frac{1}{\mu_i(\Omega)}\int_\Omega  \frac{G(\theta)^2}{\sin^2\theta}\left(1-\frac{x_i^2}{s^2}\right)\,d\omega \notag
\\
&=&\frac{1}{N\,\mu_i(\Omega)} \int_{ D_\gamma}G'(\theta)^2  \, d\omega+\frac{1}{\mu_i(\Omega)}\int_\Omega  \frac{G(\theta)^2}{\sin^2\theta}\left(1-\frac{x_i^2}{s^2}\right)\,d\omega. \label{11}
\end{eqnarray}
Summing over $i=1,...,N$ we get
\begin{equation*}\label{12}
\int_\Omega G(\theta)^2\, d\omega \le \frac{1}{N}\sum_{i=1}^N \frac{1}{\mu_i(\Omega)} \int_{ D_\gamma}G'(\theta)^2  \, d\omega+\sum_{i=1}^N\frac{1}{\mu_i(\Omega)}\int_\Omega  \frac{G(\theta)^2}{\sin^2\theta}\left(1-\frac{x_i^2}{s^2}\right)\,d\omega.
\end{equation*}
Now notice that
$$
\sum_{i=1}^N\frac{1}{\mu_i(\Omega)}\left(1-\frac{x_i^2}{s^2}\right)-\sum_{i=1}^{N-1}\frac{1}{\mu_i(\Omega)}=\frac{1}{\mu_N(\Omega)}-\sum_{i=1}^{N}\frac{1}{\mu_i(\Omega)}\,\frac{x_i^2}{s^2}\le 0,
$$
which follows from  $\mu_i(\Omega)\le \mu_N(\Omega)$ for all $i=1,...,N-1$ and the definition of $s$. Hence,
\begin{equation}\label{20}
\int_\Omega G(\theta)^2\, d\omega \le \frac{1}{N}\sum_{i=1}^N \frac{1}{\mu_i(\Omega)} \int_{ D_\gamma}G'(\theta)^2  \, d\omega+\sum_{i=1}^{N-1}\frac{1}{\mu_i(\Omega)}\int_\Omega  \frac{G(\theta)^2}{\sin^2\theta}\,d\omega.
\end{equation}
By Lemma \ref{lemma} we know that the function $\dfrac{G(\theta)}{\sin \theta}$ is decreasing in $(0,\gamma)$. Recalling that $|\Omega|=|D_\gamma|$, we get
\begin{eqnarray}
\int_\Omega  \frac{G(\theta)^2}{\sin^2\theta}\,d\omega&=& \int_{\Omega\cap D_\gamma}  \frac{G(\theta)^2}{\sin^2\theta}\,d\omega+\int_{\Omega\setminus D_\gamma}  \frac{G(\theta)^2}{\sin^2\theta}\,d\omega \notag
\\
& \le& \int_{\Omega\cap D_\gamma}  \frac{G(\theta)^2}{\sin^2\theta}\,d\omega+\frac{G(\gamma)^2}{\sin^2\gamma}\,|\Omega\setminus D_\gamma| \notag
\\
&=& \int_{\Omega\cap D_\gamma}  \frac{G(\theta)^2}{\sin^2\theta}\,d\omega+\frac{G(\gamma)^2}{\sin^2\gamma}\,| D_\gamma\setminus\Omega| \notag
\\
&\le&  \int_{\Omega\cap D_\gamma}  \frac{G(\theta)^2}{\sin^2\theta}\,d\omega+\int_{D_\gamma\setminus\Omega}\frac{G(\theta)^2}{\sin^2\theta}\,d\omega \notag
\\
&=&\int_{D_\gamma}\frac{g(\theta)^2}{\sin^2\theta}\,d\omega. \label{21}
\end{eqnarray}  
On the other side, since $G(\theta)$ is non-decreasing in $\left(0,\frac \pi 2\right)$, we have
\begin{eqnarray}
\int_\Omega G(\theta)^2\, d\omega &=& \int_{\Omega\cap D_\gamma} G(\theta)^2\, d\omega +\int_{\Omega\setminus D_\gamma} G(\theta)^2\, d\omega \notag
\\
&\ge& \int_{\Omega\cap D_\gamma} G(\theta)^2\, d\omega +G(\gamma)^2 |\Omega\setminus D_\gamma|\notag
\\
&=&\int_{\Omega\cap D_\gamma} G(\theta)^2\, d\omega +G(\gamma)^2 |D_\gamma\setminus \Omega|\notag
\\
&\ge& \int_{\Omega\cap D_\gamma} G(\theta)^2\, d\omega +\int_{D_\gamma\setminus \Omega} g(\theta)^2\, d\omega \notag
\\
&=&\int_{D_\gamma} g(\theta)^2\, d\omega.\label{22}
\end{eqnarray}
Using \eqref{20}, \eqref{21}, \eqref{22} and the monotonicity of the sequence $\{\mu_i(\Omega)\}_i$ we have
\begin{eqnarray*}
\int_{D_\gamma} g(\theta)^2\, d\omega &\le & \frac{1}{N}\sum_{i=1}^N \frac{1}{\mu_i(\Omega)}\int_{D_\gamma} g'(\theta)^2\, d\omega+\sum_{i=1}^{N-1}\frac{1}{\mu_i(\Omega)}\int_{D_\gamma}\frac{g(\theta)^2}{\sin^2\theta}\,d\omega
\\
&\le & \frac{1}{N-1}\sum_{i=1}^{N-1} \frac{1}{\mu_i(\Omega)}\int_{D_\gamma} \left[g'(\theta)^2+(N-1)\, \frac{g(\theta)^2}{\sin^2\theta}\right]\,d\omega.
\end{eqnarray*}
Finally, from \eqref{mu1} we conclude 
\begin{equation}\label{fin}
\frac{1}{N-1}\sum_{i=1}^{N-1}\frac{1}{\mu_i(\Omega)}\ge \frac{1}{\mu_1(D_\gamma)}.
\end{equation}
The equality sign holds in \eqref{fin} if and only if $\Omega$ is a geodesic ball.

\end{document}